\documentclass[12pt]{amsart} %
\usepackage{amsmath,amssymb,amsfonts}
\usepackage{amsthm}
\usepackage{graphicx}

\newtheorem{theorem}[equation]{Theorem}

\newtheorem{lemma}[equation]{Lemma}

\newtheorem{corollary}[equation]{Corollary}

\numberwithin{equation}{section}

\title{Generalized convexity of the Lambert $W$ function} %%%

\author{Gendi Wang}
\address{\textbf{Gendi Wang}\newline%({\rm *Corresponding author})
School of Science, Zhejiang Sci-Tech University, Hangzhou 310018, China}
\email{gendi.wang@zstu.edu.cn }

\begin{document} %%%
%==========================================================================
\newcounter{minutes}\setcounter{minutes}{\time}
\divide\time by 60
\newcounter{hours}\setcounter{hours}{\time}
\multiply\time by 60 \addtocounter{minutes}{-\time}
%==========================================================================
%%%%%�Զ������ļ���ʱ��
%==========================================================================
\def\thefootnote{}
\footnotetext{ {\tiny File:~\jobname.tex,
          printed: \number\year-\number\month-\number\day,
          \thehours.\ifnum\theminutes<10{0}\fi\theminutes }}
\makeatletter\def\thefootnote{\@arabic\c@footnote}\makeatother
%===============================================================================

\maketitle

%===============================================================================
\begin{abstract}
This paper investigates the generalized convexity properties of the Lambert $W$ function, defined as the solution to
$W(z)e^{W(z)}=z$. Focusing on $H_{p,q}$-convexity and concavity with respect to H\"older means, we derive necessary and sufficient conditions for $W$ to exhibit strict $H_{p,q}$-convexity or concavity on the interval $(0,+\infty)$. The main result characterizes these properties in terms of specific parameter regions $(p,q)$-plane. Inequalities involving harmonic, geometric, and arithmetic means are established, with equalities holding only when $x=y$.
\end{abstract}
%===============================================================================

%===============================================================================
{\small \sc Keywords.} {Lambert $W$ function, H\"older mean, convexity, concavity, inequalities }

{\small \sc Mathematics Subject Classification~(2010).} {33E20, 26D07}
%===============================================================================

%%%%%

%%%%%%%%%%%%%%%%%%%%%%
%%%%%%%%%%%%%%%%%%%%%%
\section{Introduction}
%%%%%%%%%%%%%%%%%%%%%%
%%%%%%%%%%%%%%%%%%%%%%

The Lambert $W$ function, defined as the solution of the equation
$W(z) e^{W(z)}=z$ for $z\in\mathbb{C}$,
 is a fundamental transcendental function with applications across mathematics, physics, and engineering.
It plays an important role in solving exponential equations, modeling growth processes, and analyzing combinatorial structures  \cite{ corless1996,mezo2022}.

When the Lambert $W$ function  is considered as a real function, its principal branch is strictly increasing from $(-\frac1e,+\infty)$ onto $(-1,+\infty)$ with $W(0)=0$ and grows logarithmically at large values of the variable.  In the last decades,  approximations of the Lambert $W$  function have been studied extensively \cite{AS2018, IB2017, Salem2020, Stewart2009}.

Despite its widespread utility, the convexity properties of the Lambert $W$ function, particularly in the context of generalized convexity, remain an area of active investigation. Generalized convexity, especially $H_{p,q}$-convexity, extends classical convexity by incorporating H\"older means, which include arithmetic, geometric, and harmonic means as special cases. Understanding such properties is crucial for optimizing inequalities and refining mathematical models in various fields \cite{AVV, GM, N, ZWC} . However, the $H_{p,q}$-convexity of the Lambert $W$ function has not been thoroughly explored, leaving gaps in the theoretical framework surrounding this function.

This paper aims to address this gap by characterizing the $H_{p,q}$-convexity and concavity of the Lambert
$W$ function on the interval $(0,+\infty)$. By deriving necessary and sufficient conditions for
$W$ to exhibit strict $H_{p,q}$-convexity or concavity, we provide a comprehensive analysis that bridges theoretical and applied mathematics. Our  main result is stated as the following theorem.

%===============================================================================
\begin{theorem}\label{arch}
~For $p,q\in \mathbb{R}$, the Lambert $W$ function is strictly $H_{p,q}$-convex on $(0,+\infty)$ if and only if $(p,q)\in D_1\cup D_2$,
while $W$ is strictly $H_{p,q}$-concave on $(0,+\infty)$ if and only if $(p,q)\in D_3$, where
$$D_1=\{(p,q)|-\infty<p\le -1,\,p\le q<+\infty\},$$
$$D_2=\{(p,q)|-1< p\le 0, \,  C(p)\le q<+\infty\},$$
$$D_3=\{(p,q)|0\le p<+\infty,\,-\infty<q\le p\},$$
and $C(p)=-2\sqrt{-p}+1$. In particular, for all $x,\,y\in(0,+\infty)$, there hold
\begin{equation}\label{h-1-1}
W\left(\frac{2xy}{x+y}\right)\le\frac{2W(x)W(y)}{W(x)+W(y)},
\end{equation}
and
\begin{equation}\label{h00}
W\left(\left(\frac{2\sqrt[4]{xy}}{\sqrt[4]{x}+\sqrt[4]{y}}\right)^4\right)\le\sqrt{W(x)W(y)}\le W(\sqrt{xy})\le \frac12(W(x)+W(y)),
\end{equation}
with equalities if and only if $x =y$.
\end{theorem}
%===============================================================================

%==========================================================================
\begin{figure}[h]
\centering
\includegraphics[width=8.2cm]{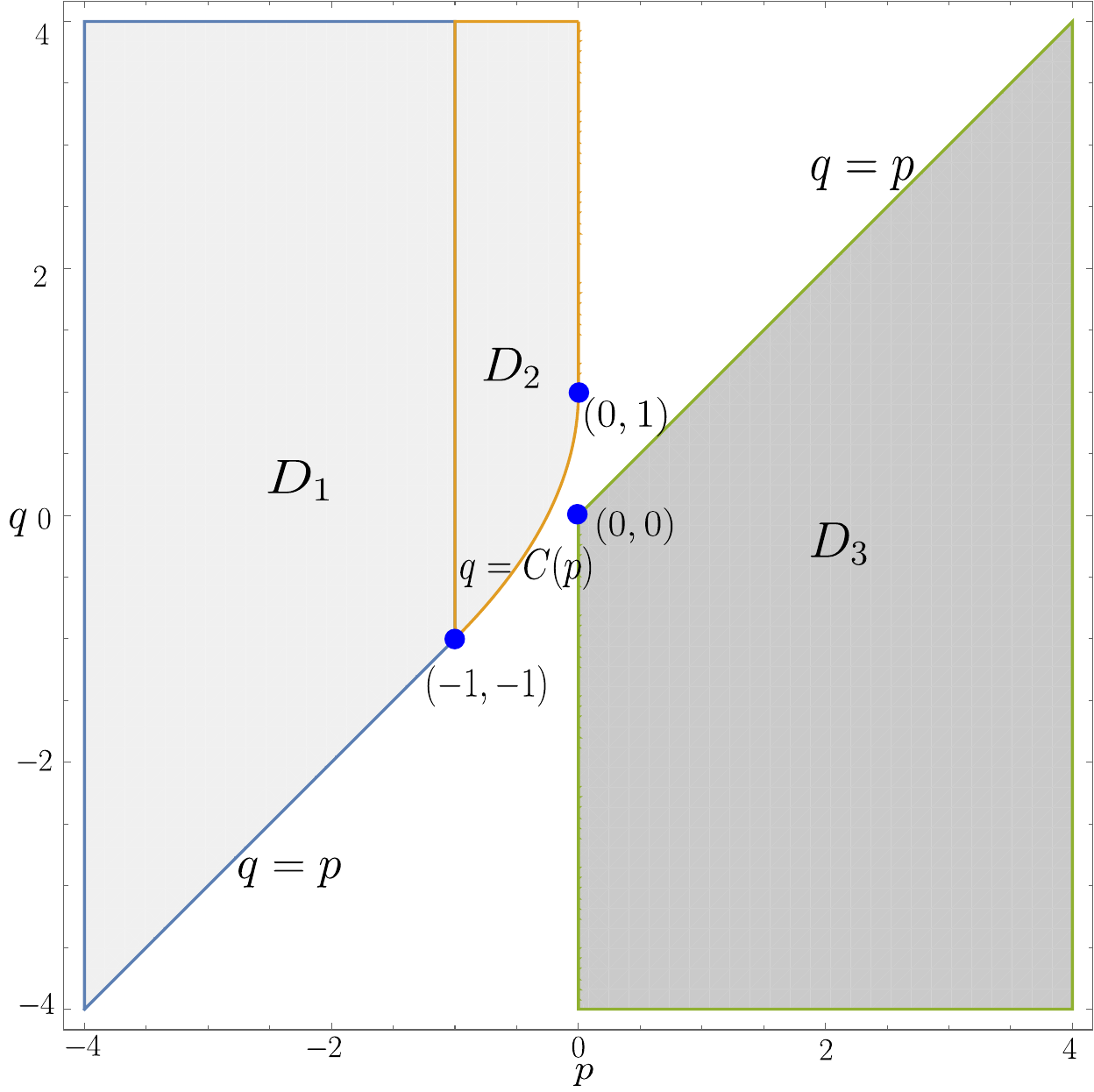}
\caption{The Lambert $W$ function is $H_{p,q}$-convex for $(p,q)\in D_1\cup D_2$ and $H_{p,q}$-concave for $(p,q)\in D_3$. }
\end{figure}
%==========================================================================

%%%%%%%%%%%%%%%%%%%%%%
%%%%%%%%%%%%%%%%%%%%%%
\section{Preliminaries}
%%%%%%%%%%%%%%%%%%%%%%
%%%%%%%%%%%%%%%%%%%%%%

For $r,s\in (0,+\infty)$, the {\it H\"older mean of order $p$} is defined by
\begin{eqnarray*}
H_p(r,s)=\Big(\frac{r^p+s^p}{2}\Big)^\frac{1}{p}\,\,\,\, {\rm for}\,\, \,p\ne 0, \,\,\,\,\,\,\,\,H_0(r,s)=\sqrt{rs}.
\end{eqnarray*}
For $p=1$, we get the arithmetic mean $A=H_1$; for $p=0$, the geometric mean $G=H_0$; and for $p=-1$, the harmonic mean $H=H_{-1}$.
It is well known that $H_p(r,s)$ is continuous and increasing with respect to $p$.

A function $f:(0,+\infty)\to(0,+\infty)$ is called {\it $H_{p,q}$-convex~(concave)} if it satisfies
\begin{equation*}
f\left(H_p(r,s)\right)\le (\ge )H_q\left(f(r),f(s)\right)
\end{equation*}
for all $r,s\in(0,\infty)$, and strictly {\it $H_{p,q}$-convex~(concave)} if the inequality is strict except for $r=s$.

The derivative of the Lambert $W$ function is
$$W'(z)=\frac{W(z)}{z(W(z)+1)}.$$

We prove the following two lemmas before giving the proof of Theorem \ref{arch}.

%===============================================================================
\begin{lemma}\label{3hr}
~For $p\in \mathbb{R}$ and $r\in (0,+\infty)$, define
$$h_p(r)=p(W(r)+1)+\frac{W(r)}{W(r)+1}.$$
(1) If $p>0$, then $h_p$ is strictly increasing with range $(p,+\infty)$.\\
(2) If $p=0$, then $h_p$ is strictly increasing with range $(0,1)$.\\
(3) If $p\le-1$, then $h_p$ is strictly decreasing with range $(-\infty,p)$.\\
(4) If $-1<p<0$, then $h_p$ is not monotone and the range of $h_p$ is $(-\infty,C(p)]$, where
$$C(p)=-2\sqrt{-p}+1.$$
%$$C(p)=\max\limits_{r\in (0,+\infty)}h_p(r)$$
%with $-1<C(p)<1$.
%Moreover,  $\lim\limits_{p\to -1}C(p)=-1$ and $\lim\limits_{p\to 0}C(p)=1$.
\end{lemma}

\begin{proof}
It is easy to get
$$h_p(0^+)=p \quad{\rm and}\quad h_p(+\infty)=\left\{
            \begin{array}{cc}
             +\infty, &p>0,\\
             1, &p=0,\\
             -\infty, &p<0.
             \end{array}
             \right.$$
By differentiation, we have
\begin{equation*}
h_p'(r)=W'(r)\left(p-f_1(r)\right),
\end{equation*}
where
$$f_1(r)=-\frac{1}{(W(r)+1)^2}.$$
Clearly, $f_1$ is strictly increasing with range $(-1,0)$.

(1)(2) If $p\ge0$, then $h_p'(r)>0$. Hence ~$h_p$ is strictly increasing with range ~$(p,+\infty)~$ if $p>0$ and $(0,1)$ if $p=0$.

(3) If $p\le-1$, then $h_p'(r)<0$. Hence ~$h_p$ is strictly decreasing with range ~$(-\infty,p)$.

(4) If $-1<p<0$, clearly $h_p$ is not monotone. Let $z=w(r)+1$, then
$$h_p(r)=p z-\frac1z+1\equiv f(z).$$
By differentiation in $z$, we have $f'(z)=p+\frac{1}{z^2}$ and  $f''(z)=-\frac{2}{z^3}$. Hence
$$C(p)=\max\limits_{r\in (0,+\infty)}h_p(r)=f\left(\frac 1{\sqrt{-p}}\right)=-2\sqrt{-p}+1,$$
the range of $h_p$ follows immediately.
\end{proof}
%===============================================================================

%===============================================================================
\noindent{\bf Remark.}
{\it Since $C(p)$ is strictly increasing and continus in $p$, 
we have $C(p)<0$ for $p\in (-1,-\frac 14)$, $C(p)>0$ for $p\in (-\frac 14,0)$ and $C(-\frac 14)=0$.}
%===============================================================================

%===============================================================================
\begin{lemma}\label{4hr}
~Let $p,\,q\in \mathbb{R}$, $r\in (0,+\infty)$, and $C(p)$ be the same as in Lemma \ref{3hr}(4). Let
\begin{equation*}
g_{p,q}(r)=\frac{W^q(r)}{r^p(W(r)+1)}.
\end{equation*}
(1) If $p>0$, then $g_{p,q}$ is strictly decreasing for each $q\le p$, and $g_{p,q}$ is not monotone for any $q>p$.\\
(2) If $p=0$, then $g_{p,q}$ is strictly increasing for each $q\ge 1$, and strictly decreasing for each $q\le 0$,
and $g_{p,q}$ is not monotone for any $0<q<1$.\\
(3) If $p\le-1$, then $g_{p,q}$ is strictly increasing for each $q\ge p$, and $g_{p,q}$ is not monotone for any $q<p$.\\
(4) If $-1<p<0$, then $g_{p,q}$ is strictly increasing for each $q\ge C(p)$, and $g_{p,q}$ is not monotone for any $q<C(p)$.
\end{lemma}
%===============================================================================
%===============================================================================
\begin{proof}
~By logarithmic differentiation in $r$, we have
\begin{equation*}
\frac{g_{p,q}'(r)}{g_{p,q}(r)}=\frac{1}{r(W(r)+1)}\left(q-h_p(r)\right),
\end{equation*}
where $h_p(r)$ is the same as in Lemma \ref{3hr}. Hence the results immediately follow from Lemma \ref{3hr}.
\end{proof}
%===============================================================================

%%%%%%%%%%%%%%%%%%%%%%
%%%%%%%%%%%%%%%%%%%%%%
\section{Proof of Main Result}
%%%%%%%%%%%%%%%%%%%%%%
%%%%%%%%%%%%%%%%%%%%%%

We are now in a position to prove Theorem \ref{arch}.

%===============================================================================
\begin{proof}[Proof of Theorem \ref{arch}]
~Without loss of generality, we may assume that $0<x\le y<+\infty$. Let $t=H_p(x,y)$, then $x\le t\le y$ and
\begin{equation*}
\frac{\partial t}{\partial x}=\frac{1}{2}\left(\frac{x}{t}\right)^{p-1}.
\end{equation*}

The proof is divided into the following four cases.

\medskip
{\bf Case 1.}~$p\ne 0$ and $q\ne 0$.

Define
\begin{equation*}
F(x,y)=W^q\,(H_p(x,y))-\frac12(W^q\,(x)+W^q\,(y)).
\end{equation*}
By differentiation, we have
\begin{equation*}
\frac{\partial F}{\partial x}=\frac{q}{2}x^{p-1}\left(\frac{W^q(t)}{t^{p}(W(t)+1)}-\frac{W^q(x)}{x^{p}(W(x)+1)}\right)
=\frac{q}{2}x^{p-1}\left(g_{p,q}(t)-g_{p,q}(x)\right),
\end{equation*}
where $g_{p,q}$ is defined in Lemma \ref{4hr}.

\medskip
{\bf Case 1.1}~$p>0$ and $q\le p$.

By Lemma \ref{4hr}(1), the function $g_{p,q}$ is strictly decreasing on $(0,+\infty)$.

{\bf Case 1.1.1}~If $q>0$, then $\frac{\partial F}{\partial x}\le 0$.
Hence $F(x,y)$ is strictly decreasing and $F(x,y)\ge F(y,y)=0$.
Namely,
\begin{equation*}
W(H_p(x,y))\ge \left(\frac{W^q(x)+W^q(y)}{2}\right)^\frac{1}{q}=H_q(W(x),W(y)),%\tag{1.1}
\end{equation*}
with equality if and only if $x=y$.

{\bf Case 1.1.2}~If $q<0$, then $\frac{\partial F}{\partial x}\ge 0$.
Hence $F(x,y)$ is strictly increasing and $F(x,y)\le F(y,y)=0$.
Namely,
\begin{equation*}
W(H_p(x,y))\ge \left(\frac{W^q(x)+W^q(y)}{2}\right)^\frac{1}{q}=H_q(W(x),W(y)),%\tag{1.1}
\end{equation*}
with equality if and only if $x=y$.

In conclusion, $W$ is strictly $H_{p,q}$-concave on the whole interval $(0,+\infty)$ for $(p,q)\in\{(p,q)|0< p<+\infty,\,0<q\le p\}\cup\{(p,q)|0< p<+\infty,\,\,q<0\}$.

\medskip
{\bf Case 1.2}~$p>0$ and $q>p$.

By Lemma \ref{4hr}(1), the function $g_{p,q}$ is not monotone on $(0,+\infty)$.
With an argument similar to Case 1.1, it is easy to see that $W$ is neither $H_{p,q}$-concave nor $H_{p,q}$-convex on the whole interval $(0,+\infty)$ for $(p,q)\in \{(p,q)|0<p<+\infty,\,q>p\}$.

\medskip
{\bf Case 1.3}~$p\le -1$ and $q\ge p$.

By Lemma \ref{4hr}(3), the function $g_{p,q}$ is strictly increasing on $(0,+\infty)$.

{\bf Case 1.3.1}~If $q>0$, then $\frac{\partial F}{\partial x}\ge 0$.
Hence $F(x,y)$ is strictly increasing and $F(x,y)\le F(y,y)=0$.
Namely,
\begin{equation*}
W(H_p(x,y))\le \left(\frac{W^q(x)+W^q(y)}{2}\right)^\frac{1}{q}=H_q(W(x),W(y)),%\tag{1.1}
\end{equation*}
with equality if and only if $x=y$.

{\bf Case 1.3.2}~If $q<0$, then $\frac{\partial F}{\partial x}\le 0$.
Hence $F(x,y)$ is strictly decreasing and $F(x,y)\ge F(y,y)=0$.
Namely,
\begin{equation*}
W(H_p(x,y))\le \left(\frac{W^q(x)+W^q(y)}{2}\right)^\frac{1}{q}=H_q(W(x),W(y)),%\tag{1.1}
\end{equation*}
with equality if and only if $x=y$.

In conclusion, $W$ is strictly $H_{p,q}$-convex on the whole interval $(0,+\infty)$ for $(p,q)\in\{(p,q)|-\infty<p\le -1,\, p\le q<0 \}\cup\{(p,q)|-\infty<p\le -1,\, 0<q<+\infty\}$.

\medskip
{\bf Case 1.4}~$p\le -1$ and $q< p$.

By Lemma \ref{4hr}(3), the function $g_{p,q}$ is not monotone on $(0,+\infty)$.
With an argument similar to Case 1.3, it is easy to see that $W$ is neither $H_{p,q}$-concave nor $H_{p,q}$-convex on the whole interval $(0,+\infty)$ for $(p,q)\in \{(p,q)|-\infty<p\le -1,\, q<p\}$.

\medskip
{\bf Case 1.5}~$-1<p< 0$ and $q\ge C(p)$.

By Lemma \ref{4hr}(4), the function $g_{p,q}$ is strictly increasing on $(0,+\infty)$.

{\bf Case 1.5.1}~If $q>0$, then $\frac{\partial F}{\partial x}\ge 0$.
Hence $F(x,y)$ is strictly increasing and $F(x,y)\le F(y,y)=0$.
Namely,
\begin{equation*}
W(H_p(x,y))\le \left(\frac{W^q(x)+W^q(y)}{2}\right)^\frac{1}{q}=H_q(W(x),W(y)),%\tag{1.1}
\end{equation*}
with equality if and only if $x=y$.

{\bf Case 1.5.2}~If $q<0$, then $\frac{\partial F}{\partial x}\le 0$.
Hence $F(x,y)$ is strictly decreasing and $F(x,y)\ge F(y,y)=0$.
Namely,
\begin{equation*}
W(H_p(x,y))\le \left(\frac{W^q(x)+W^q(y)}{2}\right)^\frac{1}{q}=H_q(W(x),W(y)),%\tag{1.1}
\end{equation*}
with equality if and only if $x=y$.

In conclusion, $W$ is strictly $H_{p,q}$-convex on the whole interval $(0,+\infty)$ for
$(p,q)\in \{(p,q)|-1<p<-\frac 14,\, C(p)\le q <0 \}
\cup \{(p,q)|-1<p\le -\frac 14,\, 0<q<+\infty\}
%\cup\{(p,q)|p=-\frac 14,\, 0<q <+\infty \}
\cup\{(p,q)| -\frac 14< p<0,\,C(p)\le q<+\infty\}$.

\medskip
{\bf Case 1.6}~$-1<p<0$ and $q<(p)$.

By Lemma \ref{4hr}(4), the function $g_{p,q}$ is not monotone on $(0,+\infty)$.
With an argument similar to Case 1.5, it is easy to see that $W$ is neither $H_{p,q}$-concave nor $H_{p,q}$-convex on the whole interval $(0,+\infty)$ for $(p,q)\in \{(p,q)|-1<p<-\frac 14,\, -\infty<q<C(p) \}
\cup\{(p,q)| -\frac 14\le p<0,\,-\infty<q<0) \}
\cup\{(p,q)|-\frac 14<p<0,\,0<q<C(p) \}$.

\medskip
{\bf Case 2.}~$p\ne 0$ and $q=0$.

Define
\begin{equation*}
F(x,y)=\frac{W^2\,(H_p(x,y))}{W(x)\cdot W(y)}.
\end{equation*}
By logarithmic differentiation, we obtain
\begin{equation*}
\frac{1}{F}\frac{\partial F}{\partial x}=x^{p-1}(g_{p,0}(t)-g_{p,0}(x)),
\end{equation*}
where $g_{p,0}$ is defined in Lemma \ref{4hr}.

\medskip
{\bf Case 2.1}~$p>0$ and $q=0$.

By Lemma \ref{4hr}(1), the function $g_{p,0}$ is strictly decreasing on $(0,+\infty)$ and hence $\frac{\partial F}{\partial x}\le 0$.
Then $F(x,y)$ is strictly decreasing and $F(x,y)\ge F(y,y)=1$.
Namely,
\begin{equation*}
W(H_p(x,y))\ge \sqrt{W(x)\cdot W(y)}=H_0(W(x),W(y)),%\tag{1.6}
\end{equation*}
with equality if and only if $x=y$.

In conclusion, $W$ is strictly $H_{p,q}$-concave on the whole interval $(0,+\infty)$ for $(p,q)\in \{(p,q)|0<p<+\infty,\,q=0\}$.

\medskip
{\bf Case 2.2}~$p\le -1$ and $q=0$.

By Lemma \ref{4hr}(3), the function $g_{p,0}$ is is strictly increasing on $(0,+\infty)$ and hence $\frac{\partial F}{\partial x}\ge 0$.
Then $F(x,y)$ is strictly increasing and $F(x,y)\le F(y,y)=1$.
Namely,
\begin{equation*}
W(H_p(x,y))\le \sqrt{W(x)\cdot W(y)}=H_0(W(x),W(y)),%\tag{1.1}
\end{equation*}
with equality if and only if $x=y$.

In conclusion, $W$ is strictly $H_{p,q}$-convex on the whole interval $(0,+\infty)$ for $(p,q)\in \{(p,q)|-\infty <p\le -1,\,q=0\}$.

\medskip
{\bf Case 2.3}~$-1<p<0$ and $q=0$.

{\bf Case 2.3.1}~$-1<p\le -\frac 14$ and $q=0$.

By Lemma \ref{4hr}(4), the function $g_{p,0}$ is strictly increasing on $(0,+\infty)$ and hence $\frac{\partial F}{\partial x}\ge 0$.
Then $F(x,y)$ is strictly increasing and $F(x,y)\le F(y,y)=1$.
Namely,
\begin{equation*}
W(H_p(x,y))\le \sqrt{W(x)\cdot W(y)}=H_0(W(x),W(y)),%\tag{1.7}
\end{equation*}
with equality if and only if $x=y$.

In conclusion, $W$ is strictly $H_{p,q}$-convex on the whole interval $(0,+\infty)$ for $(p,q)\in \{(p,q)|-1<p\le -\frac 14,\, q=0 \}$.

{\bf Case 2.3.2}~$-\frac 14<p<0$ and $q=0$.

By Lemma \ref{4hr}(4), the function $g_{p,0}$ is not monotone on $(0,+\infty)$. With an argument similar to Case 2.1-2.2, it is easy to see that $W$ is neither $H_{p,q}$-concave nor $H_{p,q}$-convex on the whole interval $(0,+\infty)$ for $(p,q)\in \{(p,q)|-\frac 14<p<0,\,q=0\}$.

\medskip
{\bf Case 3.}~$p=0$ and $q\ne 0$.

Define
\begin{equation*}
F(x,y)=W^q\,(\sqrt{xy})-\frac12\left(W^q(x)+W^q(y)\right).
\end{equation*}
By differentiation, we obtain
\begin{equation*}
\frac{\partial F}{\partial x}=\frac{q}{2x}(g_{0,q}(t)-g_{0,q}(x)),
\end{equation*}
where $g_{0,q}$ is defined in Lemma \ref{4hr}.

\medskip
{\bf Case 3.1}~$p=0$ and $q\ge 1$.

By Lemma \ref{4hr}(2), the function $g_{0,q}$ is strictly increasing on $(0,+\infty)$ and hence $\frac{\partial F}{\partial x}\ge 0$.
Then $F(x,y)$ is strictly increasing and $F(x,y)\le F(y,y)=0$.
Namely,
\begin{equation*}
W(H_0(x,y))\le \left(\frac{W^q(x)+W^q(y)}{2}\right)^\frac{1}{q}=H_q(W(x),W(y)),%\tag{1.8}
\end{equation*}
with equality if and only if $x=y$.

In conclusion, $W$ is strictly $H_{p,q}$-convex on the whole interval $(0,+\infty)$ for $(p,q)\in \{(p,q)|p=0,\,q\ge 1\}$.

\medskip
{\bf Case 3.2}~$p=0$ and $q<0$.

By Lemma \ref{4hr}(2), the function $g_{0,q}$ is strictly decreasing on $(0,+\infty)$ and hence $\frac{\partial F}{\partial x}\ge 0$.
Then $F(x,y)$ is strictly increasing and $F(x,y)\le F(y,y)=0$.
Namely,
\begin{equation*}
W(H_0(x,y))\ge \left(\frac{W^q(x)+W^q(y)}{2}\right)^\frac{1}{q}=H_q(W(x),W(y)),%\tag{1.8}
\end{equation*}
with equality if and only if $x=y$.

In conclusion, $W$ is strictly $H_{p,q}$-concave on the whole interval $(0,+\infty)$ for $(p,q)\in \{(p,q)|p=0,\,q<0\}$.

\medskip
{\bf Case 3.3}~$p=0$ and $0<q<1$.

By Lemma \ref{4hr}(2), the function $g_{0,q}$ is not monotone on $(0,+\infty)$.
With an argument similar to Case 3.1-3.2, it is easy to see that $W$ is neither $H_{p,q}$-concave nor $H_{p,q}$-convex on the whole interval $(0,+\infty)$ for $(p,q)\in \{(p,q)|p=0,\,0<q<1)\}$.

\medskip
{\bf Case 4.}~$p=0$ and $q=0$.

By Case 1.1, for all $x,y\in (0,+\infty)$, we have
\begin{equation*}
W(H_p(x,y))\ge H_p(W(x),W(y))\quad\quad {\rm for}\quad p>0.
\end{equation*}
By the continuity of $H_p$ in $p$ and $W$ in $x$, we have
\begin{equation*}
W(H_0(x,y))\ge H_0(W(x),W(y)),%\tag{2.1}
\end{equation*}
with equality if and only if $x=y$.

In conclusion, $W$ is strictly $H_{0,0}$-concave on the whole interval $(0,+\infty)$.

\medskip
By Case 1.3.2, Case 2.3.1, Case 3.1 and Case 4,  $W$ is strictly $H_{-1,\,-1}$-convex, strictly $H_{-\frac14,\,0}$-convex, strictly $H_{0,\,1}$-convex and strictly $H_{0,\,0}$-concave on $(0,+\infty)$.
Therefore, the inequalities \eqref{h-1-1}  and \eqref{h00}  hold with equalities if and only if $x =y$.

This completes the proof of Theorem \ref{arch}.
\end{proof}
%===============================================================================

Setting $p=1=q$ in Theorem \ref{arch}, we easily obtain the concavity of $W$.

%===============================================================================
\begin{corollary}
The Lambert $W$ function is strictly concave on $(0,+\infty)$.
\end{corollary}
%===============================================================================

%==============================================================================
%==============================================================================
%%%%%%%%%%%%%%%%%%%%%%%%%%%%%%%%%%%%%

\medskip

\subsection*{Acknowledgments}
 This research was supported by National Natural Science Foundation of China (NSFC) under Grant No.11771400.

%%%%%%%%%%%%%%%%%%%%%%%%%%%%%%%%%%%%%

\end{document}